\documentclass{amsart}

\newtheorem{thm}{Theorem}[section]
\newtheorem{lem}[thm]{Lemma}
\newtheorem{prop}[thm]{Proposition}
\newtheorem{coro}[thm]{Corollary}

\newtheorem*{mainthm}{Main Theorem}

\theoremstyle{definition}

\theoremstyle{remark}
\newtheorem{remark}[thm]{Remark}

\numberwithin{equation}{section}

\newcommand{\qf}[1]{\mbox{$\langle #1\rangle $}}
\newcommand{\pff}[1]{\mbox{$\langle\!\langle #1
\rangle\!\rangle $}}
\newcommand{\qpf}[1]{\mbox{$\langle\!\langle #1 ]]$}}

\newcommand{\HH}{{\mathbb H}}

\newcommand{\NN}{{\mathbb N}}

\renewcommand{\phi}{\varphi}

\DeclareMathOperator{\rad}{rad}

\DeclareMathOperator{\ani}{an}

\begin{document}

\title[Witt kernels of quadratic forms]{Witt kernels of quadratic forms 
for multiquadratic extensions in 
characteristic $2$}

\author{Detlev W.\ Hoffmann}
\address{Fakult\"at f\"ur Mathematik, Technische Universit\"at Dortmund, 
44221 Dortmund, Germany}
\email{detlev.hoffmann@math.tu-dortmund.de}

\thanks{The research on this paper has been supported in part by the 
DFG Projekt ``Annihilators and kernels in Kato's cohomology in positive 
characteristic and in Witt groups in characteristic $2$''.}


\date{}


\begin{abstract}
Let $F$ be a field of characteristic $2$ and let $K/F$ be a purely
inseparable extension of exponent $1$.  We show that the 
extension is excellent for quadratic forms.  Using the 
excellence we recover and extend results by Aravire and Laghribi 
who computed generators for the kernel $W_q(K/F)$ 
of the natural restriction map $W_q(F)\to W_q(K)$
between the Witt groups of quadratic forms of $F$ and $K$, respectively,
where $K/F$ is a finite multiquadratic extension of separability
degree at most $2$.  
\end{abstract}

\maketitle

\section{Introduction}
Throughout this article, we will only consider fields of characteristic $2$.
Let $K$ be a finite 
multiquadratic extension of a field $F$ of separability degree
at most $2$, in other words, 
$K=F(\sqrt{a_1},\cdots ,\sqrt{a_n})$, or 
$K=F(\sqrt{a_1},\cdots ,\sqrt{a_n},\wp^{-1}(b))$, $a_i,b\in F^*$, 
where $\wp^{-1}(b)$ is a root of $X^2+X+b$.
In \cite{AL}, Aravire and Laghribi computed the kernel $W_q(K/F)$ of the 
natural map (induced by scalar extension)
$W_q(F)\to W_q(K)$ between the Witt groups of nonsingular 
quadratic forms over $F$
and $K$, respectively.  They show that 
$$W_q(K/F)=\sum_{i=1}^n\qf{1,a_i}_b\otimes W_q(F)\ ,$$ in the purely
inseparable case, and 
$$W_q(K/F)=W(F)\otimes [1,b]+\sum_{i=1}^n\qf{1,a_i}_b\otimes W_q(F)$$
in the case of separability degree $2$,
where $\qf{1,a_i}_b$ is understood to be a binary bilinear form in 
its diagonal notation, $[a,b]$ represents the nonsingular quadratic form 
$ax^2+xy+by^2$, 
and $W_q(F)$ is considered as a module over the Witt ring of nonsingular
symmetric bilinear forms $W(F)$.  The proof of Aravire and Laghribi
uses differential forms. 
Actually, they prove more, namely they determine 
the kernel of the restriction map in Kato's cohomology
$H^n_2(F)\to H^n_2(K)$ and then deduce the result on the Witt kernels by
using Kato's theorem \cite{K} that $H^*_2(F)$ is naturally isomorphic to the
graded Witt module of quadratic forms over $F$.   
For more details, we refer the reader to Aravire and Laghribi's 
article \cite{AL} and the references there.

The purpose of the present paper is to give a new more
elementary proof of these results
on the Witt kernels.  Our approach is completely different.  On the 
one hand, we show less since we don't get the results on the kernels for
the graded Witt modules.  On the other hand, we show more.
Recall that an extension $K/F$ is called excellent (for quadratic forms) 
if for any
quadratic form $q$ over $F$, the anisotropic part of $q$ over $K$ is defined 
over $F$, i.e. there exists an anisotropic form $\phi$ over $F$ with
$(q_K)_{\ani}\cong\phi_K$.
We will show that the extension $K/F$ is excellent if it is
purely inseparable of exponent $1$, i.e. if $K^2\subseteq F$.
Using this, we can determine the 
Witt kernel for the compositum of such an exponent $1$ extension
with the function field of a quadratic or bilinear Pfister form.
Note that the case $K=F(\sqrt{a_1},\cdots ,\sqrt{a_n},\wp^{-1}(b))$
can be interpreted as the compositum of a finite exponent $1$ extension
with the function field of a $1$-fold quadratic Pfister form, thus
it also becomes a special case of our result.

Our main result is the following:

\begin{mainthm}\label{main-thm}
Let $K/F$ be a field extension such that $K^2\subset F$.  
\begin{enumerate}
\item[(i)] The extension  $K/F$ is excellent for quadratic forms.
\item[(ii)]
$$W_q(K/F)=\sum_{t\in K^{*2}}\qf{1,t}_b\otimes W_q(F)\ .$$
\end{enumerate}
\end{mainthm}

In the next section, we introduce some basic terminology on bilinear 
and quadratic forms in characteristic $2$.  In the third section, we prove
the main theorem, and in the fourth section, we provide some 
remarks concerning our main result and we generalize and 
extend it.

\section{Terminology and definitions}
For all undefined terminology on quadratic and bilinear forms, in particular
in characteristic $2$, we refer to \cite{EKM} and \cite{HL}.
We assume throughout that $F$ is a field of characteristic $2$.
Bilinear forms are always assumed to be symmetric, and underlying vector
spaces of bilinear and quadratic forms are always finite-dimensional.

Let $(B,V)$ be a bilinear form over an $F$-vector space $V$.  The radical
is defined to be $\rad(B)=\{ x\in V\,|\,B(x,V)=0\}$, and $B$ is said
to be nonsingular if $\rad(B)=0$.  We define the value sets
$D_F(B)=\{ B(x,x)\,|\,x\in V\setminus\{ 0\}\}$, 
$D_F^0(B)=D_F(B)\cup\{ 0\}$ and $D_F^*(B)=D_F(B)\cap F^*$.  $B$ is called
isotropic if $D_F(B)=D_F^0(B)$ and anisotropic if $D_F(B)=D_F^*(B)$.
One has the usual notions of isometry $\cong$, orthogonal sum $\perp$
and tensor product $\otimes$
of bilinear forms.  In the sequel, we always assume bilinear forms to be 
nonsingular.

A $2$-dimensional isotropic bilinear form is called a 
metabolic plane, 
in which case one can always find a basis such that the Gram matrix 
with respect to that basis is of the shape
$$\left(\begin{array}{cc} 0 & 1 \\
                          1 & a 
  \end{array}\right)$$
for some $a\in F$.  If $a=0$, this is called a hyperbolic plane.  
A bilinear form $B$ is said to be metabolic 
(resp. hyperbolic) if it is the orthogonal sum of metabolic
(resp. hyperbolic) planes.  It is not difficult to see that a 
form $B$ is hyperbolic iff $D_F(B)=\{ 0\}$.

If the Gram matrix of a form $B$ with respect to a certain basis is
a diagonal matrix with entries $a_i$, $1\leq i\leq n=\dim (B)$, then
we write $B\cong\qf{a_1,\cdots ,a_n}_b$. A diagonalization
exists iff $B$ is not hyperbolic.

A bilinear form $B$ can be decomposed as 
$B\cong B_{\ani}\perp B_m$
with $B_{\ani}$ anisotropic and $B_m$ metabolic.  $B_{\ani}$
is uniquely determined up to isometry, but generally not $B_m$.

We call two bilinear forms $B$ and $B'$ Witt equivalent 
if $B_{\ani}\cong B'_{\ani}$.
The equivalence classes together with addition induced by 
$\perp$ and  multiplication induced by $\otimes$ define the
Witt ring of $F$  denoted by $W(F)$.

An $n$-fold bilinear Pfister form (or bilinear $n$-Pfister for short)
is a form of type 
$\pff{a_1,\ldots ,a_n}_b:=\qf{1,a_1}_b\otimes\ldots\otimes\qf{1,a_n}_b$
for some $a_i\in F^*$.  A bilinear Pfister form $\pi$ is round,
i.e., $\pi\cong x\pi$ for all $x\in D_F(\pi)^*$, 
and  $\pi$ is isotropic iff $\pi$ is
metabolic.

Now let $(q,V)$ be a quadratic form over an $F$-vector space $V$,
with associated bilinear form $B_q(x,y)=q(x+y)+q(x)+q(y)$.  
One defines the radical $\rad(q)=\rad(B_q)$ and calls $q$ nonsingular
if $\rad(q)=0$ and totally singular if $\rad(q)=V$.  Isometry $\cong$ and
orthogonal sum $\perp$ are defined in the usual way, and
we define $D_F(q)=\{ q(x)\,|\,x\in V\setminus\{ 0\}\}$, 
$D_F^0(q)=D_F(q)\cup\{ 0\}$ and $D_F^*(q)=D_F(q)\cap F^*$, 
and we call $q$ 
isotropic resp. anisotropic if $D_F(q)=D_F^0(q)$ resp. $D_F(q)=D_F^*(q)$.
A nonsingular quadratic form can always be written with respect to
a suitable basis as $ax^2+xy+by^2$ with $a,b\in F$, and we denote this
form by $[a,b]$.  Totally singular forms are exactly the diagonal forms
$a_1x_1^2+\ldots +a_nx_n^2$, for which we write $\qf{a_1,\ldots,a_n}$.
A hyperbolic plane $\HH$ is a $2$-dimensional nonsingular isotropic 
quadratic form and one has $\HH\cong [0,0]$, and a quadratic form
is called hyperbolic if it is an orthogonal sum of hyperbolic planes.  
The form $[a,b]$ is hyperbolic iff $ab\in\wp(F)=\{ c^2+c\,|\,c\in F\}$. 

Each quadratic form $q$ has a decomposition $q\cong q_r\perp q_s$ with
$q_r$ nonsingular and $q_s$ totally singular.  In fact, $q_s$ is nothing
but the restriction of $q$ to $\rad(q)$ and it is thus uniquely determined,
whereas $q_r$ is generally not uniquely determined up to isometry, but
$\dim(q_r)$ is uniquely determined.  This decomposition can be 
refined as follows:
$$q\cong \underbrace{\HH\perp\ldots\perp\HH}_{k\ {\rm times}}\perp
q_0\perp \qf{a_1,\ldots,a_m}\perp
\qf{\underbrace{0,\ldots,0}_{\ell\ {\rm times}}}$$
such that $q_0$ is nonsingular and $q_{\ani}=q_0\perp \qf{a_1,\ldots,a_m}$
is anisotropic.  In this decomposition, $k$ is uniquely determined and
called the Witt index  $i_W(q)$ of $q$, $\ell$ is uniquely determined 
and called the defect  $i_d(q)$ of $q$, $q_{\ani}$ is uniquely determined
up to isometry and called the anisotropic part of $q$, and 
$(k\times\HH)\perp q_{\ani}$ is also uniquely determined
up to isometry and called the nondefective part of $q$.

If $(q,V)$ and $(q',V')$ are quadratic forms defined on $F$-vector spaces
$V$ and $V'$, respectively, then we say that $q'$ dominates $q$, $q\prec q'$,
it there exists an injective $F$-linear map $t:V\to V'$ with
$q'(tx)=q(x)$ for all $x\in V$.

We call two quadratic forms $q,q'$ Witt equivalent, $q\sim q'$, 
if $q_{\ani} \cong q'_{\ani}$.  The classes of nonsingular quadratic forms
together with addition induced by the orthogonal sum form the 
Witt group $W_q(F)$ of 
quadratic forms.  $W_q(F)$ has a natural structure
as $W(F)$-module, essentially given by scaling $\qf{a}_b\otimes q = aq$.

An $n$-fold quadratic Pfister form (or quadratic $n$-Pfister for short)
is a form of type 
$\qpf{a_1,\ldots ,a_n}:=\pff{a_1,\ldots,a_{n-1}}_b\otimes [1,a_n]$
for some $a_1,\ldots ,a_{n-1}\in F^*$, $a_n\in F$.  
A quadratic Pfister form $\pi$ is round,
i.e., $\pi\cong x\pi$ for all $x\in D_F(\pi)^*$, 
and  $\pi$ is isotropic iff $\pi$ is hyperbolic.

If $\phi$ is a (quadratic or bilinear) form over $F$ and if 
$K/F$ is a field extension, then
we write $\phi_K:=\phi\otimes K$ for the form obtained by scalar extension.  
This induces
a natural homomorphism $W(F)\to W(K)$ resp. $W_q(F)\to W_q(K)$  
whose kernel will be denoted by $W(K/F)$ resp. $W_q(K/F)$. 
In analogy to the definition of excellence of field extensions 
in the theory of quadratic forms in characteristic $\neq 2$ as defined 
in \cite{ELW}, we say that $K/F$ is excellent for quadratic resp.
bilinear forms if for any quadratic resp. bilinear form $\phi$ over
$F$ there exists a form $\psi$ over $F$ with $(\phi_K)_{\ani}\cong \psi_K$,
in other words, the anisotropic part of $\phi$ over $K$ is defined over $F$.

\section{Witt kernels and excellence for extensions of exponent one}
Let us now turn to the proof of the main theorem.  Throughout this section,
$K/F$ will be a field extension with $K^2\subset F$.  We define
$$J_{K/F}=\sum_{t\in K^{*2}}\qf{1,t}_b\otimes W_q(F)\ .$$

\begin{remark}\label{bilker}
In \cite{H1}, it was shown that a purely inseparable exponent $1$
extension $K/F$ is excellent for bilinear forms and that 
the bilinear Witt kernel $W(K/F)$ is generated
by $\{ \qf{1,t}_b\,|\,t\in K^{*2}\}$.  Hence, one readily has
$J_{K/F}=W(K/F)\otimes W_q(F)$.
\end{remark}

\begin{lem}\label{J-in-kern}
$J_{K/F}\subset W_q(K/F)$.
\end{lem}

\begin{proof}  If $t\in K^{*2}$, then clearly 
$(\qf{1,t}_b)_K\cong (\qf{1,1}_b)_K$
is metabolic and hence, for any  $q\in \qf{1,t}_b\otimes W_q(F)$ 
we have $q_K=0\in W_q(K)$.\end{proof}

\begin{lem}\label{totsing-exc} {\rm (Cf.\ \cite[Lemma 2.2]{HL}.)}
Let $q\cong q_r\perp q_s$ be a quadratic form over $F$ 
with $q_r$ nonsingular and 
$q_s\cong\qf{a_1,\ldots, a_n}$ totally
singular, not all $a_i=0$, 
and let $K/F$ be any field extension. 
Then there exist $1\leq r\leq n$ and 
$1\leq i_1<i_2<\ldots <i_t\leq n$ such that for 
$q_s'\cong \qf{a_{i_1},\ldots, a_{i_r}}$, one has 
$((q_s)_K))_{\ani}\cong (q_s')_K$. 
and $(q_K)_{\ani}\cong \phi\perp (q_s')_K$ for some nonsingular
$\phi$ over $K$. In particular,
every field extension is excellent for totally singular quadratic forms.
\end{lem}

Let us restate the main theorem in more detail.

\begin{thm}\label{main-thm-a}  
Let $q\cong q_r\perp q_s$ be an anisotropic quadratic form over $F$ with
$q_r$ nonsingular and $q_s$ totally singular.
Then there exists a nonsingular form $q_r'$ over $F$ 
with $\dim q_r'\leq\dim q_r$,
a totally singular form $q'_s\prec q_s$ over $F$, and a form
$\psi\in J_{K/F}$ such that
\begin{enumerate}
\item[(i)] $(q_K)_{\ani}\cong  (q_r'\perp q'_s)_K$, and
\item[(ii)] $q\perp\psi\sim q_r'\perp q_s$.
\end{enumerate}
In particular,  $K/F$ is excellent and $W_q(K/F)=J_{K/F}$.
\end{thm}
\begin{proof}  First note that the result on the Witt kernel follows
by considering the case $\dim q_s=0$ and $q_K\sim 0$ in which case
$\dim q_r'=0$ as well and thus $q\cong\psi$.

To prove (i) and (ii), 
write $q\cong q_r\perp q_s$ with $q_r$ nonsingular of dimension
$2m$, and $q_s$ totally singular of dimension $\ell$, and let
$n=\dim q=2m+\ell$.  

If $q_K$ is anisotropic, there is nothing to show.
Also, one readily sees that by Lemma \ref{totsing-exc}, we may assume
that if $\dim q_s>0$, then $(q_s)_K$ is anisotropic.

So let us assume that $q_K$ is isotropic, i.e.
$\dim (q_K)_{\ani}<n$, and $(q_s)_K$ is anisotropic.
In particular, $m\geq 1$.  If $m=1$ and $\ell =0$, then
$q\cong a[1,b]$ for some $b\notin\wp(F)$ and some $a\in F^*$, 
and $q_K$ being isotropic
means that $a[1,b]_K\cong\HH$ and hence $b\in \wp(K)$ and thus
$F(\wp^{-1}(b))\subset K$, a contradiction to $K/F$ being purely
inseparable.  

Hence, we may assume $m\geq 1$ and $n=2m+\ell\geq 3$.

To prove the theorem, it suffices to construct a form
$\beta\in J_{K/F}$ such that for 
$\widetilde{q}=(q\perp\beta)_{\ani}$ one has 
$\dim\widetilde{q}<\dim q$.  Then we know that
$q_K\sim \widetilde{q}_K$ and we conclude by a simple
induction on $m$.

Let us write
$$q\cong \underbrace{[a_1,b_1]\perp\ldots\perp [a_m,b_m]}_{q_r}
\perp \underbrace{\qf{c_1,\ldots,c_\ell}}_{q_s}\ .$$
Since $q_K$ is isotropic, there exist $x_i,y_i,z_j\in K$,
$1\leq i\leq m$, $1\leq j\leq\ell$, with
\begin{equation}\label{sum1}
0=\sum_{i=1}^m (a_ix_i^2+x_iy_i+b_iy_i^2)+\sum_{j=1}^\ell c_jz_j^2\ 
\end{equation}
and not all $x_i,y_i,z_j$ equal to $0$.

Suppose there exists $i\in\{1,\ldots,m\}$
with $x_i=y_i=0$, or $j\in \{1,\ldots,\ell\}$ with $z_j=0$.
Then write $q\cong \widehat{q}\perp [a_i,b_i]$ resp.
$q\cong \widehat{q}\perp\qf{c_i}$.
Thus, $\dim \widehat{q}<\dim q$ and $\widehat{q}_K$ is isotropic.
We then may proceed by working with $\widehat{q}$ rather than
$q$.

So we may assume that for each $1\leq i\leq m$ we have 
$x_i\neq 0$ or $y_i\neq 0$,
and that for each $1\leq j\leq \ell$ we have  $z_j\neq 0$.
Without loss of generality, there exists $1\leq k\leq m$ with
$x_i\neq 0$ for $1\leq i\leq k$, and $x_{k+1}=\ldots =x_m=0$.
In particular, $y_i\neq 0$ for $k+1\leq i\leq m$.

Note that $K^2\subseteq F$, so $x_i^2,y_i^2,z_j^2\in F$, so,
if $x_i\neq 0$,
$\pi_i\cong\qf{1,x_i^2}_b\otimes [a_i,b_i]\in J_{K/F}$ and
$[a_i,b_i]\perp\pi_i\sim [a_ix_i^2,\frac{b_i}{x_i^2}]$.  Put
$$q_1:= \bigl[a_1x_1^2,\frac{b_1}{x_1^2}\bigr]\perp\ldots\perp
\bigl[a_kx_k^2,\frac{b_k}{x_k^2}\bigr]\perp 
[a_{k+1},b_{k+1}]\perp\ldots\perp [a_m,b_m]
\perp q_s\ .$$
Then $q_1\sim q  \perp\pi_1\perp\ldots\perp\pi_{k}$,
$\dim q=\dim q_1$ and $q_K\sim (q_1)_K$. 
Since we are interested in the anisotropic part of $q$ over $K$ and
since $q_1$ differs from $q$ by a form in $J_{K/F}\subseteq
W_q(K/F)$, and since
$$
a_ix_i^2+x_iy_i+b_iy_i^2=(a_ix_i^2)\cdot 1^2
+1\cdot (x_iy_i)+\frac{b_i}{x_i^2}(x_iy_i)^2\ ,$$
we may assume furthermore, by replacing $q$ by $q_1$, that
in Eq. \ref{sum1}  and by reassigning the letters $a_i$, $b_i$, 
we have $x_1=\ldots =x_k=1$, so together with the
other assumptions
$$0=\sum_{i=1}^k(a_i+y_i+b_iy_i^2)+\sum_{i=k+1}^m b_iy_i^2+
\sum_{j=1}^\ell c_jz_j^2\ .$$
Now using repeatedly $[t,u]\perp [v,w]\cong [t+v,u]\perp [v,u+w]$, we get
\begin{equation}\label{sum2}
[a_1,b_1]\perp\ldots \perp[a_k,b_k]\cong
[a_1+\ldots +a_k,b_1]\perp [a_2,b_1+b_2]\perp\ldots\perp [a_k,b_1+b_k]\ .
\end{equation}
Note that furthermore
\begin{equation}\label{sum3}
\begin{array}{rcl}
0 & = & (a_1+\ldots +a_k)+(y_1+\ldots +y_k)+ b_1(y_1+\ldots +y_k)^2\\[1ex]
 &  & + (b_1+b_2)y_2^2+\ldots + (b_1+b_k)y_k^2+\sum_{i=k+1}^m b_iy_i^2+
\sum_{j=1}^\ell c_jz_j^2\ ,
\end{array}
\end{equation}
so in view of Eqs. \ref{sum2} and \ref{sum3}, we may assume $k=1$, i.e. 
in Eq. \ref{sum1}, we have $x_1=1$, $x_2=\ldots =x_m=0$.
We also still may assume that $y_i\neq 0$ for $2\leq i\leq m$ or 
else we could again omit those terms $[a_i,b_i]$ with $i\geq 2$ and $y_i=0$.  
By adding forms in $J_{K/F}$ similarly as above, we may furthermore
assume $y_2=\ldots=y_m=1$, so that Eq. \ref{sum1} becomes
\begin{equation}\label{sum4}
0=a_1+y_1+b_1y_1^2+\sum_{i=1}^k b_i+\sum_{j=1}^\ell c_jz_j^2\ .\end{equation}
Note that, again similarly as before,
\begin{equation}\label{sum5}
[a_2,b_2]\perp\ldots \perp[a_m,b_m]\cong
[a_2,b_2+\ldots+b_m]\perp [a_2+a_3,b_3]\perp\ldots\perp [a_2+a_m,b_m]
\end{equation}
Eqs. \ref{sum4} and \ref{sum5} together imply that in Eq. \ref{sum1}, 
we may assume
$m\leq 2$, $x_1=1$ and $y_2=1$ (if $m=2$), still with $2m+\ell\geq 3$.

Suppose $\ell =0$.  Then $m=2$ and
we are in the case $q\cong [a_1,b_1]\perp [a_2,b_2]$ and there
exists $y\in K$ with $a_1+y+b_1y^2+b_2=0$.  But $y^2\in F$, therefore
also $y\in F$ and we have that $q$ is also isotropic over $F$, a 
contradiction.

Suppose now $\ell >0$ and $m=2$.  
So $q\cong [a_1,b_1]\perp [a_2,b_2]\perp\qf{c_1,\ldots c_\ell}$ and 
Eq. \ref{sum1} becomes with the above assumptions
$0=a_1+y+b_1y^2+b_2+c_1z_1^2+\ldots +c_\ell z_\ell^2$ with $y,z_i\in K$.
Recall that by assumption all $z_i\in K^*$, hence $z_i^2\in F^*$.
Then put
$$\tau_i:=\qf{1,z_i^2}_b\otimes [a_2z_i^2,c_i]\cong 
[a_2z_i^2,c_i]\perp [c_iz_i^2,a_2]\in J_{K/F}$$
and we have 
$$[a_2,d]\perp\qf{c_i}\perp\tau_i\sim [a_2,d+c_iz_i^2]\perp\qf{c_i}$$
and hence, with $\tau\cong\tau_1\perp\ldots\perp\tau_\ell\in J_{K/F}$,
we get
$$[a_2,b_2]\perp q_s\perp\tau
\sim [a_2,b_2+c_1z_1^2+\ldots +c_\ell z_\ell^2]\perp q_s\ .$$
Put $b_2'=b_2+c_1z_1^2+\ldots +c_\ell z_\ell^2$.
This shows that for $q_2:=[a_1,b_1]\perp [a_2,b_2']\perp q_s$ we have
$q\perp\tau\sim q_2$, hence 
$q_K\sim (q_2)_K$ and (with the same $y$ as before)
$a_1+y+b_1y^2+b_2'=0$, so $q_2$ is isotropic, and
$\widetilde{q}:=(q_2)_{\ani}$ is the desired form with
$\dim\widetilde{q}<\dim q_2=\dim q$.

Suppose finally that $\ell >0$ and $m=1$, so 
$q\cong [a_1,b_1]\perp\qf{c_1,\ldots c_\ell}$ and 
Eq. \ref{sum1} becomes with the above assumptions
$0=a_1+y+b_1y^2+c_1z_1^2+\ldots +c_\ell z_\ell^2$ with $y,z_i\in K$.

If $y=0$, put 
$$\rho_i:=\qf{1,z_i^2}_b\otimes [b_1z_i^2,c_i]\cong
[b_1z_i^2,c_i]\perp [b_1,c_iz_i^2]\in J_{K/F}\ ,$$
and we have
$$[d,b_1]\perp\qf{c_i}\perp\rho_i\sim [d+c_iz_i^2,b_1]\perp\qf{c_i}\ ,$$
and with $\rho:=\rho_1\perp\ldots\perp\rho_\ell\in J_{K/F}$, 
we get similarly as before
$$q\perp\rho\sim [\underbrace{a_1+c_1z_1^2+\ldots +c_\ell z_\ell^2}_0,b_1]
\perp q_s\sim q_s\ .$$
and $\widetilde{q}:=q_s$ is the desired form.

If $y\neq 0$, put 
$$\nu_i:=\bigl\langle 1,\frac{z_i^2}{y^2}\bigr\rangle_b\otimes 
\bigl[a_1\frac{z_i^2}{y^2},c_i\bigr]
\cong \bigl[a\frac{z_i^2}{y^2},c_i\bigr]\perp 
\bigl[a_1,c_i\frac{z_i^2}{y^2}\bigr]\in J_{K/F}\ ,$$
and we have
$$[a_1,d]\perp\qf{c_i}\perp\nu_i\sim 
\bigl[a_1,d+c_i\frac{z_i^2}{y^2}\bigr]\perp\qf{c_i}\ ,$$
and with $\nu:=\nu_1\perp\ldots\perp\nu_\ell\in J_{K/F}$
and  $b_1':=b_1+c_1\frac{z_1^2}{y^2}+
\ldots +c_\ell \frac{z_\ell^2}{y^2}$ we get
$$q\perp\nu\sim [a_1,b_1']\perp q_s$$
and (with the same $y$ as before) $a_1+y+b_1'y^2=0$, i.e. $[a_1,b_1']$ is
isotropic and hence hyperbolic, and
$q\perp\nu\sim q_s$, and again $\widetilde{q}:=q_s$ is the desired form. 
This completes the proof.
\end{proof}

\section{Some corollaries and remarks}
We are now interested in extensions that are obtained
by composing purely inseparable exponent $1$ extensions 
with function fields of Pfister
forms, and in the Witt kernels of these new extensions.

If $\pi$ is an $n$-fold quadratic Pfister form over 
$F$ ($n\geq 1$),
then let $F(\pi)$ be the function field of the projective quadric
$X_\pi$ given by the equation $\pi=0$.  

If $B=\pff{a_1,\ldots ,a_n}_b$, $a_i\in F^*$, is an
$n$-fold bilinear Pfister form  with associated 
totally singular quadratic form $q(x)=B(x,x)$, i.e.
$q=\pff{a_1,\ldots ,a_n}$, then there exist $k\leq n$ and
$1\leq i_1<\ldots <i_k\leq n$ with 
$q_{\ani}\cong\pff{a_{i_1},\ldots,a_{i_k}}$ (see, e.g., \cite[\S 8]{HL}). 
Put $B_0=\pff{a_{i_1},\ldots,a_{i_k}}_b$. We define the function field
$F(B):=F(q)$.  Then $F(q)$ is a purely transcendental extension
of $F(q_{\ani})=F(B_0)$ (see, e.g., \cite[Remark 7.4(iii)]{H}).  
Since anisotropic forms stay anisotropic over 
purely transcendental extensions, we have that $W_q(F(B)/F)=W_q(F(B_0)/F)$.
Note that if $k=0$, i.e. $q_{\ani}\cong\qf{1}$ and thus $B_0\cong\qf{1}_b$,
we have $F(B_0)=F$ and hence
$W_q(F(B)/F)=W_q(F(B_0)/F)=0$.

We have the following results on Witt kernels for function fields of
Pfister forms.

\begin{prop}\label{pfker}
Let $n\geq 1$ and let $\pi$ is an $n$-fold quadratic Pfister form over 
$F$, or let $B=\pff{a_1,\ldots ,a_n}_b$, $a_i\in F^*$, be an
$n$-fold bilinear Pfister form with  associated 
totally singular quadratic form $q(x)=B(x,x)$ such that
$\dim q_{\ani}\geq 2$, i.e. $[F^2(a_1,\ldots ,a_n):F^2]>1$.
\begin{enumerate}
\item[(i)] {\rm (\cite[Th.\ 1.4(2)]{L} or
\cite[Cor.\ 23.6]{EKM}.)} Let $\phi\in W_q(F(\pi)/F)$
be anisotropic, then there exists a bilinear form $\beta$ with
$\phi\cong\beta\otimes\pi$.  In particular,
$W_q(F(\pi)/F)=W(F)\otimes\pi$.
\item[(ii)] {\rm (\cite[Th.\ 1.4(3)]{L}.)} 
Let $\phi$ be an anisotropic form in $W_q(F(B)/F)$, then
there exists a nonsingular quadratic form $\tau$ with
$\phi\cong B_0\otimes\tau$.  In particular, 
$W_q(F(B)/F)=B_0\otimes W_q(F)$.
\end{enumerate}
\end{prop}

\begin{remark}\label{quadker}  Note that for $b\in F\setminus \wp(F)$, 
$\pi =[1,b]$ is an anisotropic $1$-fold 
quadratic Pfister form and  $F(\pi)=F(\wp^{-1}(b))$,
and for $a\in F\setminus F^2$, $B=\qf{1,a}_b$ is an anisotropic
$1$-fold bilinear Pfister form and $F(B)=F(\sqrt{a})$, 
so the kernels in Proposition \ref{pfker} are nothing but the
well known ones for separable resp. inseparable quadratic extensions.
\end{remark}

\begin{coro}\label{quadpfister} Let $n\geq 1$ and let 
$\pi$ be an $n$-fold quadratic Pfister 
form over $F$.  Let $K$ be a purely inseparable extension of $F$
of exponent $1$.  Let $L=K(\pi)$.  Then
$W_q(L/F)=W_q(K/F)+WF\otimes \pi$.
\end{coro}
\begin{proof}  
Clearly, $W_q(K/F)+WF\otimes \pi\subseteq W_q(L/F)$.
Conversely, let $\phi$ be a form over $F$ with $\phi\in W_q(L/F)$.
By the excellence property of $K/F$, there exists a form
$\psi$ over $F$ with $(\phi_K)_{\ani}\cong \psi_K$.  If $\dim\psi =0$ then
$\phi\in W_q(K/F)$ and
we are done.  So suppose $\dim\psi >0$.  Since $0=\phi_L=\psi_L\in W_q(L)$,
it follows by Proposition \ref{pfker}
that there exists a bilinear form $\beta$ over $L$ with
$\psi_L\cong\beta\otimes\pi_L$.  By the excellence of $K/F$, we may
in fact assume that $\beta$ is already defined over $F$, see 
\cite[Prop.\ 2.11]{ELW} (the argument there works also in 
characteristic $2$), so we may put $\psi\cong\beta\otimes\pi$.  But
then $\phi -\beta\otimes\pi\in W_q(K/F)$ and hence
$\phi\in W_q(K/F)+WF\otimes \pi$ as desired.
\end{proof}

\begin{lem}\label{mamo} Let $K/F$ be a purely 
inseparable extension of exponent $1$,
let $B$ be a bilinear Pfister form over $F$ and let $\rho$ be a nonsingular
quadratic form over $K$ such that $B_K\otimes \rho$ is defined over $F$,
i.e. there exists a quadratic form $\phi$ over $F$ with
$\phi_K\cong B_K\otimes \rho$.  Then there exists a nonsingular quadratic
form $\rho_0$ over $F$ with 
$\phi_K\cong B_K\otimes \rho\cong (B\otimes\rho_0)_K$.
\end{lem}

\begin{proof}  This is essentially \cite[Th.\ 1]{MM}, except that
there $K$ was assumed to be an inseparable quadratic extension.
But the proof works in exactly the same way by using the fact that
for any $b\in K$, one has that $[1,b]$ is defined over $F$ since
$[1,b]\cong [1,b^2]$ over $K$ with $b^2\in F$ since $K^2\subset F$.
\end{proof}

\begin{coro}\label{bilpfister} Let $K$ be a purely inseparable extension of $F$
of exponent $1$. Let $n\geq 1$ and let $B=\pff{a_1,\ldots ,a_n}_b$, 
$a_i\in F^*$, be an $n$-fold bilinear Pfister 
form over $F$ with associated totally singular quadratic form
$q=\pff{a_1,\ldots ,a_n}$, and let $k\leq n$ and
$1\leq i_1<\ldots <i_k\leq n$ be such that 
$(q_K)_{\ani}\cong\pff{a_{i_1},\ldots,a_{i_k}}_K$. 
Put $B_0=\pff{a_{i_1},\ldots,a_{i_k}}_b$.  Let $L=K(B)$.
If $\dim B_0=1$ then  $W_q(L/F)=W_q(K/F)$.
If $\dim B_0>1$ then
$W_q(L/F)=W_q(K/F)+B_0\otimes W_q(F)$.
\end{coro}
\begin{proof}  By the remarks preceding Proposition \ref{pfker}, 
we have that $K(B)/K(B_0)$  is purely transcendental, hence
$W_q(L/F)=W_q(K(B_0)/F)$.  We are done if $\dim B_0=1$ as then
$K(B_0)=K$.  So let us assume $\dim B_0>1$.  Clearly,
$W_q(K/F)+B_0\otimes W_q(F)\subseteq W_q(L/F)$.  Conversely,
let $\phi$ be a form over $F$ with $\phi\in W_q(L/F)=W_q(K(B_0)/F)$.
By the excellence property of $K/F$, there exists a form
$\psi$ over $F$ with $(\phi_K)_{\ani}\cong \psi_K$.  If $\dim\psi =0$
then $\phi\in W_q(K/F)$ and we are done.  So suppose $\dim\psi >0$.  Since 
$0=\phi_{K(B_0)}=\psi_{K(B_0)}\in W_q(K(B_0))$,
it follows from Proposition \ref{pfker}
that there exists a nonsingular form
$\tau$ over $K$ with $\psi_K\cong (B_0)_K\otimes\tau$.
By Lemma \ref{mamo}, we may assume that $\tau$ is already defined
over $F$, hence we may put $\psi\cong B_0\otimes\tau$. But
then $\phi -B_0\otimes\tau\in W_q(K/F)$ and hence
$\phi\in W_q(K/F)+B_0\otimes W_q(F)$ as desired.
\end{proof}

Theorem \ref{main-thm-a} essentially says that $W_q(K/F)$ is additively
generated by forms $\qf{1,t}_b\otimes [c,d]$ with $t\in K^{*2}$ and
$c,d\in F$.  We next show that we don't need all $t\in K^{*2}$ for this
to be true.  Recall that a subset $T\subset F$ is called $2$-independent
if for all finite subsets $\{ t_1,\ldots,t_n\}\subseteq T$ with $n\in\NN$
and $t_i\neq t_j$ for $i\neq j$, one has
$$[F(\sqrt{t_1},\ldots,\sqrt{t_n}):F]=[F^2(t_1,\ldots, t_n):F^2]=2^n\ .$$
A $2$-independent set $T$ with $F^2(T)=F$ is called a $2$-basis of $F$.

If $K/F$ is a purely inseparable exponent $1$ extension, then clearly
there exists a $2$-independent set $T\subseteq F$ with 
$K=F(\sqrt{t}\,|\,t\in T)$.

\begin{coro}\label{2-ind} Let $K/F$ be a purely inseparable 
exponent $1$ extension, and let $T\subseteq F$ be a $2$-independent 
set such that $K=F(\sqrt{t}\,|\,t\in T)$.  Then
$$W_q(K/F)=\sum_{t\in T}\qf{1,t}_b\otimes W_q(F)\ .$$
\end{coro}
\begin{proof}  Obviously, 
$\sum_{t\in T}\qf{1,t}_b\otimes W_q(F)\subseteq W_q(K/F)$ as $t\in T$
implies $t\in K^{*2}$. 

For the reverse inclusion,
let $q\in W_q(K/F)$ be anisotropic.  It is clear that
there exists already a finite subset $T'\subset T$ such that for
$K'=F(\sqrt{t}\,|\,t\in T')$, we have $q\in W_q(K'/F)$.
By invoking Proposition \ref{pfker} and Remark \ref{quadker}
together with Corollary \ref{bilpfister}, 
an easy induction on the cardinality $|T'|$ shows that
$W_q(K'/F)= \sum_{t\in T'}\qf{1,t}_b\otimes W_q(F)$.
Hence $q\in W_q(K'/F)\subseteq \sum_{t\in T}\qf{1,t}_b\otimes W_q(F)$ 
as desired.
\end{proof}

Corollary \ref{2-ind} together with Corollary \ref{quadpfister} (in the case
$n=1$, see Remark \ref{quadker}) immediately imply the result by
Aravire and Laghribi \cite{AL} mentioned in the introduction.

\begin{coro} Let $K=F(\sqrt{a_1},\ldots ,\sqrt{a_n})$, $a_i\in F^*$,
and $L=K(\wp^{-1}(b))$ for some $b\in F$.   Then
$$\begin{array}{rcl}
W_q(K/F) & = & \sum_{i=1}^n\qf{1,a_i}_b\otimes W_q(F)\\[1ex]
W_q(L/F) & = & W(F)\otimes [1,b] +\sum_{i=1}^n\qf{1,a_i}_b\otimes W_q(F)
\end{array}$$
\end{coro}

\begin{remark}  A different way of
deriving the previous corollary for $W_q(K/F)$
from our Theorem \ref{main-thm-a} is as follows.  It suffices to show
that each form of type $\qf{1,t}_b\otimes [c,d]$, $t\in K^{*2}$, 
$c,d\in F$, can be written in $W_q(F)$ 
as a  sum of forms of type $\qf{1,a_i}_b\otimes q_i$ for suitable
nonsingular quadratic forms $q_i$ over $F$.

Now $t\in F^2(a_1,\ldots,a_n)$, so for each $I\subseteq\{ 1,\ldots, n\}$
there exist $x_I\in F$ such that
$t=\sum_{I\subset \{ 1,\ldots, n\}}(\prod_{i\in I}a_i)x_I^2$.
The desired result now follows by a straightforward 
induction on $n$ using the following
relations in $W_q(F)$, where we assume $w\in F$, $u,v,x\in F^*$
with $u+v\neq 0$ in the third relation:
\begin{itemize}
\item $\qf{1,1}_b\otimes [1,w]=0$;
\item $\qf{1,ux^2}_b\otimes [1,w]=\qf{1,u}_b\otimes [1,w]$;
\item $\qf{1,u+v}_b\otimes [1,w]=
\qf{1,u}_b\otimes [1,\frac{wu}{u+v}]+\qf{1,v}_b\otimes [1,\frac{wv}{u+v}]$;
\item $\qf{1,uv}_b\otimes [1,w]=\qf{1,u}_b\otimes [1,w]+
\qf{1,v}_b\otimes [u,\frac{w}{u}]$.
\end{itemize}
\end{remark}

\bibliographystyle{amsplain}

\end{document}